\documentclass[11pt,reqno]{amsart}
\usepackage{amsthm,amssymb,amsmath}

\usepackage{xcolor}

\usepackage{fix-cm}
\usepackage{enumitem,needspace}
\usepackage[normalem]{ulem}
\usepackage{esint}
\usepackage{mlmodern}
\usepackage[expansion=false]{microtype}

\usepackage[T1]{fontenc}

\usepackage[colorlinks,citecolor=blue,hypertexnames=false]{hyperref}

\AtBeginDocument{
  \hypersetup{
    urlcolor=blue,
    citecolor=red,
    linkcolor=blue,
  }
}

\newtheorem{theorem}{Theorem}
\newtheorem{proposition}{Proposition}
\newtheorem*{theorem*}{Theorem}
\newtheorem{lemma}{Lemma}[section]
\newtheorem{corollary}{Corollary}

\theoremstyle{definition}

\newtheorem*{definition*}{\bf Definition}

\newtheorem*{remark*}{\sc Remark}

\newtheorem*{example*}{\bf Example}

\newcommand{\Rd}{\mathbb R^d}
\newcommand{\Cc}{C_c^\infty(\Rd)}
\newcommand{\dd}{\,dx}

\DeclareMathOperator{\Rea}{Re}

\numberwithin{equation}{section}

\begin{document}

\title[The $L^1$-realization of an elliptic operator]{The $L^1$-realization of an elliptic operator with unbounded coefficients}
\author{Sallah Eddine Boutiah}



\begin{abstract}
We establish an $L^1$-theory for a class of elliptic operators with unbounded coefficients, extending our earlier semigroup generation results and those of Canale--Rhandi--Tacelli to the endpoint $p=1$. The $L^1$-realization generates a positive analytic $C_0$-semigroup that is compact for every positive time. We characterize the generator domain through weighted estimates and prove consistency with the corresponding $L^p$-semigroups for $1<p<\infty$.
\end{abstract}

\address{Universit\'{e} Laval, D\'{e}partement de math\'{e}matiques et de statistique, Qu\'{e}bec, QC, Canada and Laboratoire de Math\'{e}matiques Appliqu\'{e}es, Universit\'{e} Ferhat Abbas, S\'{e}tif 1, Campus El Bez, S\'{e}tif, Algeria}
\email{sallah-eddine.boutiah.1@ulaval.ca}
\email{boutiah@univ-setif.dz}
\fontsize{10.5pt}{4.4mm}\selectfont
\maketitle


\section{Introduction and main result}
\label{sec:introduction}

Partial differential equations with singular or unbounded coefficients arise in mathematical physics, for example in the study of Schr\"odinger operators with confining potentials and in diffusion models with spatially varying diffusivity.

In this paper, we extend to $L^1(\mathbb R^d)$ the analytic semigroup theory for the elliptic operator
\begin{equation}\label{def-A}
\Lambda=(1+|x|^\alpha)\Delta
+b|x|^{\alpha-2}x\cdot\nabla-c|x|^\beta,
\end{equation}
where $d\geq3$, $\alpha>2$, $\beta>\alpha-2$, $b\in\mathbb R$ and $c>0$.
We prove that the closure of $(\Lambda,C_c^\infty(\mathbb R^d))$ in $L^1(\mathbb R^d)$ generates a positive analytic
$C_0$-semigroup and characterize its generator domain explicitly. The corresponding results for $1<p<\infty$ were established by Boutiah, Gregorio, Rhandi and Tacelli~\cite{boutiah et al}, following the work of Canale, Rhandi and Tacelli~\cite{CRT} on the Schr\"odinger-type operator corresponding to the case $b=0$ and $c=1$.

An $L^1$ theory is needed to obtain a more complete solution theory for this operator, since it allows us to treat integrable initial data that belong to no $L^p$ space with $p>1$. For example, a radial analogue of the function used by Cruz-Uribe and Fiorenza~\cite{CUF} is
\[
f(x)=\frac{1}{|x|^d(\log |x|)^2}\,
\chi_{(0,e^{-1})}(|x|),
\]
with $f$ defined to be zero elsewhere. Integration in polar coordinates shows that $f\in L^1(\mathbb R^d)$, whereas $f\notin L^p(\mathbb R^d)$ for every $p>1$. Our result therefore yields solutions for such initial data, with continuous dependence on the initial datum in the $L^1$ norm.

Two difficulties arise when extending these results to $p=1$:
\begin{itemize}
\item[-] The Calder\'on--Zygmund estimate used to control all second-order derivatives by the Laplacian fails in $L^1$.
We need different estimates to characterize the generator domain.

\item[-] The proof of analyticity in~\cite{boutiah et al} relies on a gradient term with coefficient $p-1$.
This term vanishes at $p=1$, so the same argument no longer provides the estimate needed to prove analyticity.
\end{itemize}

The condition $\beta>\alpha-2$ ensures that the potential $c|x|^\beta$ grows faster at infinity than the terms of order
$|x|^{\alpha-2}$ arising from the diffusion and the drift. A Kato inequality and weighted interpolation then give
separate bounds for the potential term, the weighted gradient and the weighted Laplacian in terms of the graph norm.
These estimates lead to an explicit characterization of the generator domain.

We also prove that the resulting semigroup is consistent with the $L^p$-semigroups for $1<p<\infty$ and compact for every $t>0$. We determine the smallest exponent $\omega$ for which $\|T_1(t)\|_{\mathcal L(L^1)}\leq e^{\omega t}$ holds for all $t\geq0$.

To state our results, define

\begin{align}\label{domain_realz}
 \mathcal{D}(\Lambda):=\bigl\{u\in W^{1,1}(\Rd):\;&
 \Delta u\in L^1_{\mathrm{loc}}(\Rd),\\
 &|x|^\beta u,\;
 (1+|x|^{\alpha-1})\nabla u,\;
 (1+|x|^\alpha)\Delta u
 \in L^1(\Rd)\bigr\},\nonumber
\end{align}

where derivatives are understood in the distributional sense, and set
\begin{equation}\label{omega1}
 \omega_1  :=\sup_{x\in\Rd}  \left\{ (\alpha-b)(d+\alpha-2)|x|^{\alpha-2} -c|x|^\beta  \right\}.
\end{equation}
The assumptions $\beta>\alpha-2$ and $c>0$ ensure that
$\omega_1$ is finite.

\medskip

\begin{theorem}\label{thm:main}
Let $d\geq 3$, $\alpha>2$, $\beta>\alpha-2$, $b\in\mathbb{R}$ and $c>0$. Then,  the following assertions hold.
\begin{enumerate}[label=(\roman*),leftmargin=*,itemsep=3pt,topsep=3pt]
\item The operator $(\Lambda,\Cc)$ is closable in $L^1(\Rd)$. Its closure $(\Lambda,\mathcal{D}(\Lambda))$ generates a positive strongly continuous semigroup $(T_1(t))_{t\geq0}$ satisfying
\begin{equation}\label{semigroup-bound}
 \|T_1(t)\|_{\mathcal L(L^1)}
 \leq e^{\omega_1t},
 \qquad t\geq0.
\end{equation}
The resolvent of this realization is positive for every real $\lambda>\omega_1$, and
\begin{equation}\label{resolvent-bound}
 (\omega_1,\infty)\subset\rho(\Lambda),\qquad
 \|R(\lambda,\Lambda)\|_{\mathcal L(L^1)}
 \leq(\lambda-\omega_1)^{-1}.
\end{equation}
Moreover, $\Cc$ is a core for $(\Lambda,\mathcal{D}(\Lambda))$.

\item Estimate \eqref{graph-estimate} holds on $\mathcal{D}(\Lambda)$. Consequently, the norm
\begin{equation}\label{domain-norm}
\begin{aligned}
 \|u\|_{\mathcal{D}(\Lambda)}
 &:=\|u\|_1+\||x|^\beta u\|_1\\
 &\quad+\|(1+|x|^{\alpha-1})\nabla u\|_1
       +\|(1+|x|^\alpha)\Delta u\|_1
\end{aligned}
\end{equation}
is equivalent to the graph norm
$\|u\|_1+\|\Lambda u\|_1$.

\item For every $1<p<\infty$,
\begin{equation}\label{consistency}
 T_1(t)f=T_p(t)f,
 \qquad f\in L^1(\Rd)\cap L^p(\Rd),\quad t\geq0,
\end{equation}
where $T_p$ is the semigroup generated by $\Lambda_p$
in \cite{boutiah et al}.
\end{enumerate}
\end{theorem}

The following corollary gives the optimal exponential bound and establishes compactness of the resolvent.

\begin{corollary}
\label{cor:optimal-bound-compact-resolvent}
Under the assumptions of Theorem~\ref{thm:main},
the following assertions hold.
\begin{enumerate}
\item[(i)]
Set 
$$\gamma:=(\alpha-b)(d+\alpha-2).$$
 Then
\begin{equation}\label{omega1-explicit}
 \omega_1=
 \begin{cases}
  0, & b\geq\alpha,\\[2mm]
  \displaystyle
  \frac{\beta-\alpha+2}{\beta}\,\gamma
  \left(\frac{(\alpha-2)\gamma}{c\beta}\right)^{
  \frac{\alpha-2}{\beta-\alpha+2}},
  & b<\alpha.
 \end{cases}
\end{equation}
Moreover, $\omega_1$ is the smallest real number $\omega$ such that
\begin{equation}\label{optimal-exponential-bound}
 \|T_1(t)\|_{\mathcal L(L^1)}\leq e^{\omega t},
 \qquad t\geq0.
\end{equation}
In particular, $(T_1(t))_{t\geq0}$ is contractive on
$L^1(\mathbb R^d)$ if and only if $b\geq\alpha$.

\item[(ii)]
The domain $\mathcal D(\Lambda)$, endowed with the graph norm, is compactly embedded in $L^1(\mathbb R^d)$.
Consequently, $R(\lambda,\Lambda)$ is compact on $L^1(\mathbb R^d)$ for every $\lambda\in\rho(\Lambda)$.
\end{enumerate}
\end{corollary}

We now establish analyticity of the semigroup constructed in Theorem~\ref{thm:main}. Together with the compactness
of the resolvent, this implies that $T_1(t)$ is compact for every $t>0$.

\begin{theorem}
\label{thm:L1-analyticity}
Under the assumptions of Theorem~\ref{thm:main}, the semigroup $(T_1(t))_{t\geq0}$ is analytic on $L^1(\mathbb R^d)$,
and $T_1(t)$ is compact for every $t>0$.

More precisely, there exist $\theta\in(\pi/2,\pi)$, $\omega>\omega_1$ and $M>0$ such that
\[
\omega+\Sigma_\theta\subset\rho(\Lambda),
\qquad
\|R(\lambda,\Lambda)\|_{\mathcal L(L^1)}
\leq\frac{M}{|\lambda-\omega|},
\quad \lambda\in\omega+\Sigma_\theta,
\]
where
$\Sigma_\theta =\{z\in\mathbb C\setminus\{0\}:|\arg z|<\theta\}.$
\end{theorem}

\subsection{Related results}
\label{subsec:related-results}

The operator \eqref{def-A} and several of its special cases have been studied in $L^p$ spaces. For $b=c=0$, Metafune and Spina \cite{MS} established analytic semigroup generation for $d\geq3$, $\alpha>2$ and $ \frac{d}{d-2}<p<\infty$. They also investigated the generator domain and spectral properties. Metafune and Spina~\cite{MS14} also studied the degenerate operator $|x|^\alpha\Delta$. For $d\geq3$ and $\alpha>2$, they established analytic semigroup generation in
$L^p(\mathbb R^d)$ for $d/(d-2)<p<\infty$ and showed that this restriction is sharp.

The case $c=0$ was considered by Metafune, Spina and Tacelli \cite{MST}. For $\alpha>2$, they obtained analytic semigroup generation for $1<p<\infty$ satisfying $b>2-d$ and $p>\frac{d}{d-2+b}$ as well as domain characterizations and kernel estimates. Thus, without the potential term, the admissible range of $p$ depends on the dimension and the drift coefficient.
Metafune, Okazawa, Sobajima and Spina~\cite{MOSS} studied scale invariant elliptic operators with diffusion
$|x|^\alpha$, radial drift and a potential proportional to $|x|^{\alpha-2}$. They determined the parameter ranges
for semigroup generation in $L^p(\mathbb R^d)$, $1<p<\infty$, and characterized the generator domain.
Here, the diffusion is $1+|x|^\alpha$ and the potential has exponent $\beta>\alpha-2$.

For $b=0$ and $c=1$, Canale, Rhandi and Tacelli \cite{CRT} studied the Schr\"odinger-type operator
$
 (1+|x|^\alpha)\Delta-|x|^\beta.
$
Under the assumptions $d\geq3$, $\alpha>2$ and $\beta>\alpha-2$, they proved analytic semigroup
generation for every $1<p<\infty$ and characterized the generator domain through weighted Sobolev estimates.

The full operator \eqref{def-A}, with arbitrary $b\in\mathbb R$ and $c>0$, was studied by Boutiah, Gregorio, Rhandi and Tacelli \cite{boutiah et al}. For $1<p<\infty$, they established positive analytic semigroup generation, domain characterization and consistency, together with compactness and ultracontractivity.
Subsequently, Boutiah, Rhandi and Tacelli \cite{BRT} obtained heat kernel bounds and eigenfunction estimates. Complementary lower bounds were established by Boutiah \cite{B}, using the Davies--Simon method
in a weighted symmetric setting. These bounds exhibit the same spatial decay as the upper
estimates.

Lorenzi and Rhandi~\cite{LR} studied $(1+|x|^\alpha)\Delta-|x|^\beta$ for $0\leq\alpha\leq2$ and $\beta\geq0$, establishing analytic semigroup generation in $L^p(\mathbb R^d)$, $1<p<\infty$. They also characterized the generator domain and derived heat kernel estimates under additional assumptions.

A different regime was investigated by Boutiah, Caso, Gregorio and Tacelli \cite{BCGT}. For $0\leq\alpha<2$, they considered an additional potential proportional to $|x|^{\alpha-2}$. Using form methods and perturbation arguments,
they established generation results and conditions identifying the generators with closures of operators defined on smooth functions.

Fornaro and Lorenzi~\cite{FL} established generation and domain results in $L^1$ for operators of the form $a\Delta+F\cdot\nabla$. Their assumptions imply $a(x)\le C(1+|x|^2)$ and require $\Delta a$ to be bounded. These conditions are used in their estimates and in the proof of analyticity. They fail for $a(x)=1+|x|^\alpha$ when $\alpha>2$, since $a$ grows faster than $|x|^2$ and $\Delta a=\alpha(d+\alpha-2)|x|^{\alpha-2}$ is unbounded. Their $L^1$ theorem therefore does not apply directly to~\eqref{def-A}.

Sobajima~\cite{S} established $m$-accretivity and $m$-sectoriality results for second-order elliptic
operators with unbounded coefficients in $L^p(\mathbb R^d)$, $1<p<\infty$. His analysis includes a limiting case of the growth conditions on the coefficients, while the present work addresses the endpoint $p=1$.

At the endpoint $p=1$, generation and domain results for operators with constant diffusion and unbounded
lower-order coefficients were obtained by Metafune, Pr\"uss, Rhandi and Schnaubelt \cite[Section~5]{MPRS}.

Ouhabaz \cite[Theorem 2.4]{O} proved \(L^1\)-analyticity for self-adjoint semigroups satisfying Gaussian domination (which fails here because the diffusion coefficient is unbounded). His applications to elliptic operators assume bounded principal coefficients \cite[Theorem 3.1]{O}. These results do not directly apply to our operator $\Lambda$, which is also nonsymmetric in $L^2(\mathbb{R}^d,dx)$ when $b\ne\alpha$.

\medskip

\noindent
\textbf{Notation.}
All function spaces are considered over $\mathbb C$, with respect to Lebesgue measure. For $1\leq p\leq\infty$, we write
$L^p=L^p(\mathbb R^d)$ and denote its norm by $\|\cdot\|_p$. We denote by $C_c^\infty(\mathbb R^d)$ the space of smooth functions with compact support, and by $C_0(\mathbb R^d)$ the space of continuous functions vanishing at infinity.
The Sobolev spaces are denoted by $W^{k,p}(\mathbb R^d)$.  All derivatives are understood in the distributional sense.
For $x\in\mathbb R^d$ and $r>0$, we set $ B(x,r):=\{y\in\mathbb R^d:|y-x|<r\}$, and $B_r:=B(0,r)$. Complex conjugation and real parts are denoted by $\overline{u}$ and $\operatorname{Re}u$, respectively.

For a Banach space $X$, $\mathcal L(X)$ denotes the space of bounded linear operators on $X$. For a closed operator
$A$ on $X$, we write $\rho(A)$ for its resolvent set and
$
 R(\lambda,A):=(\lambda I-A)^{-1}$,
$\lambda\in\rho(A)$, where $I$ is the identity operator. For $0<\theta<\pi$, we set
$
 \Sigma_\theta
 :=\{z\in\mathbb C\setminus\{0\}:|\arg z|<\theta\},
$
where $\arg z$ is the principal argument.

The letter $C$ denotes a positive constant whose value may change from line to line. Dependence on additional
parameters is indicated when needed.

\subsection*{Acknowledgements}
I thank Damir Kinzebulatov for discussions regarding this paper.

\section{Weighted estimates and domain characterization}
\label{sec:weighted-estimates}

The following Kato-type inequality shows that $\Lambda-\omega_1 I$, with domain $C_c^\infty(\mathbb R^d;\mathbb C)$, is dissipative in $L^1(\mathbb R^d)$.

\begin{lemma}\label{kato}
For every $u\in C_c^\infty(\mathbb R^d;\mathbb C)$, with the
convention $\overline u/|u|=0$ on $\{u=0\}$, one has
\begin{equation}\label{kato-est}
 \Rea\int_{\mathbb R^d}
 \frac{\overline u}{|u|}\Lambda u\dd
 \leq
 \int_{\mathbb R^d}
 \bigl((\alpha-b)(d+\alpha-2)|x|^{\alpha-2}
       -c|x|^\beta\bigr)|u|\dd.
\end{equation}
Consequently, for every real $\lambda>\omega_1$,
\[
 (\lambda-\omega_1)\|u\|_1
 \leq \|(\lambda-\Lambda)u\|_1.
\]
\end{lemma}

\begin{proof}
For $\varepsilon>0$, set $ v_\varepsilon:=\sqrt{|u|^2+\varepsilon^2}-\varepsilon$. Then $v_\varepsilon\in C_c^\infty(\mathbb R^d)$,
$\operatorname{supp}v_\varepsilon\subseteq\operatorname{supp}u$,
and $0\leq v_\varepsilon\leq|u|$.
The chain rule gives
\begin{equation}\label{nabla-v}
 \nabla v_\varepsilon =\frac{\Rea(\overline u\,\nabla u)}{\sqrt{|u|^2+\varepsilon^2}}
\end{equation}
and
\begin{align*}
 \Delta v_\varepsilon
 -\frac{\Rea(\overline u\,\Delta u)}
       {\sqrt{|u|^2+\varepsilon^2}}
 &\geq
 \frac{|\nabla u|^2}{\sqrt{|u|^2+\varepsilon^2}}
 -\frac{|\Rea(\overline u\,\nabla u)|^2}
       {(|u|^2+\varepsilon^2)^{3/2}}
 \geq
 \frac{\varepsilon^2|\nabla u|^2}
      {(|u|^2+\varepsilon^2)^{3/2}}
 \geq0,
\end{align*}
where we used the Cauchy--Schwarz inequality $ |\Rea(\overline u\,\nabla u)|^2 \leq |u|^2|\nabla u|^2$. We have,
$$
 \frac{|u|^2}{\sqrt{|u|^2+\varepsilon^2}} -v_\varepsilon
 =
 \varepsilon -\frac{\varepsilon^2}{\sqrt{|u|^2+\varepsilon^2}} \geq 0.
$$
Together with \eqref{nabla-v} and the positivity of $1+|x|^\alpha$ and $c$, these yield
\[
 \Rea\left(
 \frac{\overline u}{\sqrt{|u|^2+\varepsilon^2}}\Lambda u
 \right)
 \leq \Lambda v_\varepsilon.
\]
Since $v_\varepsilon$ is smooth and compactly supported, integration by parts gives
\begin{align*}
 \Rea\int_{\mathbb R^d}
 \frac{\overline u}{\sqrt{|u|^2+\varepsilon^2}}
 \Lambda u\dd
 &\leq \int_{\mathbb R^d}\Lambda v_\varepsilon\dd\\
 &=\int_{\mathbb R^d}
 \Bigl[
 \Delta(1+|x|^\alpha)
 -b\,\operatorname{div}(|x|^{\alpha-2}x)
 -c|x|^\beta
 \Bigr]v_\varepsilon\dd\\
 &=\int_{\mathbb R^d}
 \bigl((\alpha-b)(d+\alpha-2)|x|^{\alpha-2}
       -c|x|^\beta\bigr)v_\varepsilon\dd.
\end{align*}
To pass to the limit, observe that
$$
\left|
 \frac{\overline u}{\sqrt{|u|^2+\varepsilon^2}}\Lambda u
 \right|
 \leq|\Lambda u|\in L^1(\mathbb R^d),
$$
and 
$$
\left|
 \bigl((\alpha-b)(d+\alpha-2)|x|^{\alpha-2}
       -c|x|^\beta\bigr)v_\varepsilon
 \right| \leq
 \bigl(|\alpha-b|(d+\alpha-2)|x|^{\alpha-2}
       +c|x|^\beta\bigr)|u|
 \in L^1(\mathbb R^d).
$$
Since, pointwise as $\varepsilon\downarrow0$, $v_\varepsilon \to |u|$, and  $\frac{\overline u}{\sqrt{|u|^2+\varepsilon^2}}
 \to \frac{\overline u}{|u|}$, dominated convergence gives \eqref{kato-est}. Finally, \eqref{kato-est} and the definition of $\omega_1$
imply that, for every real $\lambda>\omega_1$,
\begin{align*}
 \|(\lambda-\Lambda)u\|_1
 &\geq
 \Rea\int_{\mathbb R^d}
 \frac{\overline u}{|u|}(\lambda-\Lambda)u\dd\\
 &=\lambda\|u\|_1
 -\Rea\int_{\mathbb R^d}
 \frac{\overline u}{|u|}\Lambda u\dd\\
 &\geq(\lambda-\omega_1)\|u\|_1.
\end{align*}
\end{proof}

The following weighted estimates yield an explicit description of the closure of $(\Lambda,\Cc)$ in $L^1(\Rd)$.

\begin{proposition}\label{closure}
The operator $(\Lambda,\Cc)$ is closable in $L^1(\mathbb R^d)$.
Its closure is $(\Lambda,\mathcal{D}(\Lambda))$. Moreover, there exists $C>0$ such that
\begin{equation}\label{graph-estimate}
 \||x|^\beta u\|_1
 +\|(1+|x|^{\alpha-1})\nabla u\|_1
 +\|(1+|x|^\alpha)\Delta u\|_1
 \leq C\bigl(\|\Lambda u\|_1+\|u\|_1\bigr),
 \qquad u\in \mathcal{D}(\Lambda).
\end{equation}
\end{proposition}
\begin{proof}
The proof is divided into four steps. We first establish a weighted interpolation estimate. We then combine it with the Kato inequality to obtain the graph estimate on $C_c^\infty(\mathbb R^d)$. The third step proves closability.
Finally, approximation by cutoffs and mollification identifies the domain of the closure and shows that
$C_c^\infty(\mathbb R^d)$ is a core.

Throughout the proof, $C>0$ denotes a constant independent of the functions being estimated.

\medskip
\noindent
\emph{Step 1.}
We claim that, for every $\varepsilon>0$, there exists
$C_\varepsilon>0$ such that
\begin{equation}\label{weighted-interpolation}
 \|(1+|x|^{\alpha-1})\nabla u\|_1
 \leq \varepsilon\|(1+|x|^\alpha)\Delta u\|_1
      +C_\varepsilon\|(1+|x|^{\alpha-2})u\|_1,
 \qquad u\in\Cc.
\end{equation}

Let $P_t:=e^{t\Delta}$ be the heat semigroup. For $v\in\Cc$ and $t>0$,
\[
 v=P_tv-\int_0^tP_s\Delta v\,ds,
 \qquad
 \|\nabla P_sf\|_1\leq C_d s^{-1/2}\|f\|_1.
\]
Consequently,
\[
 \|\nabla v\|_1
 \leq C_d t^{-1/2}\|v\|_1
      +2C_d t^{1/2}\|\Delta v\|_1.
\]
For any $\rho>0$, choosing $t=(\rho/(2C_d))^2$ gives
\begin{equation}\label{unweighted}
 \|\nabla v\|_1
 \leq \rho\|\Delta v\|_1
      +2C_d^2\rho^{-1}\|v\|_1,
 \qquad v\in\Cc.
\end{equation}
Choose a radial function $\varphi\in\Cc$, nonincreasing in $|x|$,
such that
\[
0\leq\varphi\leq1,\qquad
\varphi=1\ \text{on }B_1,\qquad
\varphi=0\ \text{outside }B_2.
\]
Set
\[
\theta_0(x):=\varphi(x),\qquad
\theta_j(x):=\varphi(2^{-j}x)-\varphi(2^{1-j}x),
\quad j\geq1.
\]
The monotonicity of $\varphi$ gives $\theta_j\geq0$. Summing the definitions, we obtain
$
\sum_{j=0}^N\theta_j(x)=\varphi(2^{-N}x).
$
For each fixed $x$, we have $2^{-N}x\in B_1$ when $N$ is sufficiently large. Since $\varphi=1$ on $B_1$, it follows that
$
\sum_{j\geq0}\theta_j(x)=1
$
for every $x\in\mathbb R^d$, this shows that $(\theta_j)_{j\geq0}$ is a smooth partition of unity. Moreover,
\[
\operatorname{supp}\theta_0\subseteq\{|x|\leq2\},
\qquad
\operatorname{supp}\theta_j
\subseteq\{2^{j-1}\leq|x|\leq2^{j+1}\},
\quad j\geq1.
\]
Each point belongs to at most three of these supports. For $j\geq1$, these inclusions imply
\[
2^{j-1}\leq1+|x|\leq1+2^{j+1}\leq3\,2^j
\quad\text{on }\operatorname{supp}\theta_j.
\]
The same bounds hold for $j=0$, since $|x|\leq2$ on $\operatorname{supp}\theta_0$.

For the derivatives, the chain rule gives, for $j\geq1$,
\[
\begin{aligned}
\nabla\theta_j(x)
&=2^{-j}\nabla\varphi(2^{-j}x)
  -2^{1-j}\nabla\varphi(2^{1-j}x),\\
\Delta\theta_j(x)
&=2^{-2j}\Delta\varphi(2^{-j}x)
  -2^{2-2j}\Delta\varphi(2^{1-j}x).
\end{aligned}
\]
Since $\nabla\varphi$ and $\Delta\varphi$ are bounded and $\theta_0=\varphi$, we obtain, for every $j\geq0$,
\begin{equation}\label{dyadic-bounds}
\begin{gathered}
2^{j-1}\leq1+|x|\leq3\,2^j
\quad\text{on }\operatorname{supp}\theta_j,\\
|\nabla\theta_j(x)|\leq C2^{-j},
\qquad
|\Delta\theta_j(x)|\leq C2^{-2j}
\quad\text{for all }x\in\mathbb R^d,
\end{gathered}
\end{equation}
where $C$ depends only on $\varphi$.  Let $0<\delta\leq1$. Applying \eqref{unweighted} to $\theta_j u$ with $\rho=\delta2^j$, we obtain
\[
\|\nabla(\theta_j u)\|_1
\leq
\delta2^j\|\Delta(\theta_j u)\|_1
+C\delta^{-1}2^{-j}\|\theta_j u\|_1.
\]
Since $\sum_{j\geq0}\theta_j=1$, we have
\[
\nabla u=\sum_{j\geq0}\nabla(\theta_j u).
\]
Thus, \eqref{dyadic-bounds} gives
\begin{align*}
\|(1+|x|^{\alpha-1})\nabla u\|_1
&\leq
C\sum_{j\geq0}2^{j(\alpha-1)}
\|\nabla(\theta_j u)\|_1\\
&\leq
C\delta\sum_{j\geq0}2^{j\alpha}
\|\Delta(\theta_j u)\|_1
+C\delta^{-1}\sum_{j\geq0}2^{j(\alpha-2)}
\|\theta_j u\|_1.
\end{align*}
We now use 
$
\Delta(\theta_j u) = \theta_j\Delta u +2\nabla\theta_j\cdot\nabla u +u\Delta\theta_j.
$
The derivative bounds in \eqref{dyadic-bounds}, together with the fact that each point belongs to at most three
supports, yield
\begin{align}
\|(1+|x|^{\alpha-1})\nabla u\|_1
&\leq
C\delta\|(1+|x|^\alpha)\Delta u\|_1
+C\delta\|(1+|x|^{\alpha-1})\nabla u\|_1
\notag\\
&\quad+
C\delta^{-1}\|(1+|x|^{\alpha-2})u\|_1.
\label{localized-interpolation}
\end{align}
Here we used $\delta\leq1$, and $C$ is independent of $\delta$ and $u$. Given $\varepsilon>0$, choose
\[
0<\delta\leq
\min\left\{1,\frac{1}{2C},\frac{\varepsilon}{2C}\right\}.
\]
Moving the gradient term on the right-hand side of \eqref{localized-interpolation} to the left and using
$1-C\delta\geq\frac12$, we obtain
\[
\|(1+|x|^{\alpha-1})\nabla u\|_1
\leq
2C\delta\|(1+|x|^\alpha)\Delta u\|_1
+2C\delta^{-1}\|(1+|x|^{\alpha-2})u\|_1.
\]
Since $2C\delta\leq\varepsilon$, this proves \eqref{weighted-interpolation} with
$C_\varepsilon=2C\delta^{-1}$.

\medskip
\noindent
\emph{Step 2.}
Let $u\in\Cc$. Since $\beta>\alpha-2$ and $c>0$, there exists $C_0>0$ such that
\[
 (\alpha-b)(d+\alpha-2)|x|^{\alpha-2}
 \leq \frac c2|x|^\beta+C_0.
\]
Lemma~\ref{kato} therefore gives
\[
 \frac c2\||x|^\beta u\|_1
 \leq C_0\|u\|_1
 -\Rea\int_{\mathbb R^d}
       \frac{\overline u}{|u|}\Lambda u\dd
 \leq C_0\|u\|_1+\|\Lambda u\|_1.
\]
Since $1+|x|^{\alpha-2}\leq C(1+|x|^\beta)$, we obtain
\begin{equation}\label{potential-control}
 \||x|^\beta u\|_1
 +\|(1+|x|^{\alpha-2})u\|_1
 \leq C\bigl(\|\Lambda u\|_1+\|u\|_1\bigr).
\end{equation}

The definition of $\Lambda$ implies
\begin{equation}\label{laplacian-control}
 \|(1+|x|^\alpha)\Delta u\|_1
 \leq \|\Lambda u\|_1
      +|b|\|(1+|x|^{\alpha-1})\nabla u\|_1
      +c\||x|^\beta u\|_1.
\end{equation}
Combining \eqref{weighted-interpolation}, \eqref{potential-control} and \eqref{laplacian-control}, we find
\[
 \|(1+|x|^{\alpha-1})\nabla u\|_1
 \leq
 \varepsilon|b|\|(1+|x|^{\alpha-1})\nabla u\|_1
 +C_\varepsilon\bigl(\|\Lambda u\|_1+\|u\|_1\bigr).
\]
Taking $\varepsilon=[2(1+|b|)]^{-1}$  yields
\begin{equation}\label{gradient-control}
 \|(1+|x|^{\alpha-1})\nabla u\|_1
 \leq C\bigl(\|\Lambda u\|_1+\|u\|_1\bigr).
\end{equation}
Substituting \eqref{potential-control} and
\eqref{gradient-control} into \eqref{laplacian-control}
proves \eqref{graph-estimate} for $u\in\Cc$.

\medskip
\noindent
\emph{Step 3.}
Suppose that $ u_n\in\Cc$, $u_n\rightarrow u$, and $\Lambda u_n\rightarrow f$ in $L^1(\mathbb R^d)$. Applying \eqref{graph-estimate} to $u_n-u_m$ shows that $|x|^\beta u_n$, $(1+|x|^{\alpha-1})\nabla u_n$, and $(1+|x|^\alpha)\Delta u_n$ are Cauchy sequences in $L^1$. Since the derivative weights are bounded below by one,
$\nabla u_n$ and $\Delta u_n$ also converge in $L^1$. Their limits are $\nabla u$ and $\Delta u$,
respectively. Thus $ u\in W^{1,1}(\mathbb R^d)$ and $\Delta u\in L^1(\mathbb R^d)$. Local boundedness of the weights identifies the weighted limits:
\begin{equation}\label{weighted-limits}
 \begin{aligned}
  |x|^\beta u_n
  &\rightarrow |x|^\beta u,\\
  (1+|x|^{\alpha-1})\nabla u_n
  &\rightarrow (1+|x|^{\alpha-1})\nabla u,\\
  (1+|x|^\alpha)\Delta u_n
  &\rightarrow (1+|x|^\alpha)\Delta u
 \end{aligned}
 \qquad\text{in }L^1.
\end{equation}
Consequently, $u\in\mathcal D(\Lambda)$. By \eqref{weighted-limits}, each term in $\Lambda u_n$
converges in $L^1$ to the corresponding term in $\Lambda u$. Hence $f=\Lambda u$.

In particular, $u=0$ implies $f=0$, so $(\Lambda,\Cc)$ is closable and
\begin{equation}\label{closure-inclusion}
 \overline{(\Lambda,\Cc)}
 \subseteq(\Lambda,\mathcal D(\Lambda)).
\end{equation}

\medskip
\noindent
\emph{Step 4.}
Let $u\in\mathcal D(\Lambda)$. Using the function $\varphi$ chosen in Step~1, set
\[
 \varphi_R(x):=\varphi(x/R),
 \qquad
 u_R:=\varphi_Ru,
 \qquad R\geq1.
\]
The product rule gives $u_R\in\mathcal D(\Lambda)$ and
\begin{equation}\label{cutoff-commutator}
 \begin{aligned}
  \Lambda u_R-\varphi_R\Lambda u
  &=2(1+|x|^\alpha)\nabla\varphi_R\cdot\nabla u\\
  &\quad+
  \bigl((1+|x|^\alpha)\Delta\varphi_R
        +b|x|^{\alpha-2}x\cdot\nabla\varphi_R\bigr)u.
 \end{aligned}
\end{equation}
The derivatives of $\varphi_R$ are supported in $\{R\leq|x|\leq2R\}$ and satisfy
$
 |\nabla\varphi_R|\leq CR^{-1}$ and 
 $
 |\Delta\varphi_R|\leq CR^{-2}.
$
It follows from \eqref{cutoff-commutator} that
\begin{align*}
 \|\Lambda u_R-\Lambda u\|_1
 &\leq \|(\varphi_R-1)\Lambda u\|_1\\
 &\quad+
 C\int_{\{R\leq|x|\leq2R\}}
 \bigl((1+|x|^{\alpha-1})|\nabla u|
       +(1+|x|^{\alpha-2})|u|\bigr)\dd
 \rightarrow 0.
\end{align*}
The integrand belongs to $L^1$ because $u\in\mathcal D(\Lambda)$ and
$1+|x|^{\alpha-2}\leq C(1+|x|^\beta)$. Since also $u_R\to u$ in $L^1$, we obtain
\begin{equation}\label{cutoff-convergence}
 \|u_R-u\|_1+\|\Lambda u_R-\Lambda u\|_1
 \longrightarrow0
 \qquad(R\to\infty).
\end{equation}
Fix $R\geq1$. Choose $\eta\in C_c^\infty(B_1)$ with
$\eta\geq0$ and $\int_{\mathbb R^d}\eta\dd=1$, and set
\[
 \eta_\delta(x):=\delta^{-d}\eta(x/\delta),
 \qquad
 u_{R,\delta}:=\eta_\delta*u_R.
\]
For $0<\delta<1$, we have
$
 u_{R,\delta}\in\Cc$ and
 $
 \operatorname{supp}u_{R,\delta}\subset B_{2R+1}.
$
Since $u_R$, $\nabla u_R$ and $\Delta u_R$ belong to $L^1$, mollification gives
\begin{equation}\label{mollifier-limits}
 \begin{aligned}
  u_{R,\delta}&\rightarrow u_R,\\
  \nabla u_{R,\delta} &\rightarrow\nabla u_R,\\
  \Delta u_{R,\delta} &\rightarrow\Delta u_R
 \end{aligned}
 \qquad\text{in }L^1
 \quad(\delta\downarrow0).
\end{equation}
The coefficients of $\Lambda$ are bounded on $B_{2R+1}$. Thus \eqref{mollifier-limits} implies
\begin{equation}\label{mollification-convergence}
 \begin{aligned}
  &\|u_{R,\delta}-u_R\|_1
   +\|\Lambda u_{R,\delta}-\Lambda u_R\|_1\\
  &\quad\leq C_R\bigl(
   \|u_{R,\delta}-u_R\|_1
   +\|\nabla u_{R,\delta}-\nabla u_R\|_1
   +\|\Delta u_{R,\delta}-\Delta u_R\|_1
   \bigr)
   \rightarrow0.
 \end{aligned}
\end{equation}

By \eqref{mollification-convergence}, for each $n\geq1$ we may choose $0<\delta_n<1/n$ such that $u_{n,\delta_n}\in\Cc$ satisfies
\[
 \|u_{n,\delta_n} -\varphi_nu\|_1
 +\|\Lambda u_{n,\delta_n}-\Lambda \varphi_nu\|_1
 <\frac{1}{n}.
\]
Combining this estimate with \eqref{cutoff-convergence} yields
\begin{equation}\label{core-approximation}
 u_{n,\delta_n} \in \Cc,\qquad
 \| u_{n,\delta_n} -u\|_1+\|\Lambda u_{n,\delta_n} -\Lambda u\|_1
 \longrightarrow0.
\end{equation}
This proves the reverse inclusion in
\eqref{closure-inclusion}, and therefore
\[
 \overline{(\Lambda,\Cc)}
 =(\Lambda,\mathcal D(\Lambda)).
\]

Finally, \eqref{core-approximation} gives the weighted convergences \eqref{weighted-limits} with $u_n=u_{n,\delta_n}$.
Passing to the limit in \eqref{graph-estimate} for $u_{n,\delta_n}$ proves the estimate for every $u\in\mathcal D(\Lambda)$.
\end{proof}

We next show that, for sufficiently large $\lambda$, the $L^p$ resolvents constructed in \cite{boutiah et al}
map $\Cc$ into $\mathcal D(\Lambda)$ for every $1<p<\infty$.

\begin{lemma}\label{range}
There exists $\lambda_*>\omega_1$ such that
\[
R(\lambda,\Lambda_p)f\in \mathcal{D}(\Lambda),
\qquad
\lambda\geq\lambda_*,\quad f\in C_c^\infty(\mathbb R^d),
\quad 1<p<\infty.
\]
For fixed $\lambda$ and $f$, these resolvent values are independent of $p$ and are nonnegative whenever $f\geq0$.
\end{lemma}

\begin{proof}
By \cite[Theorems~1--2]{boutiah et al}, we may choose
$\lambda_*>\max\{\omega_1,0\}$ such that, for every
$\lambda\geq\lambda_*$ and $f\in\Cc$,
$
 u:=R(\lambda,\Lambda_p)f\in C_0(\mathbb R^d)
$
is independent of $p$ and is nonnegative whenever $f\geq0$. It remains to prove that $u\in\mathcal D(\Lambda)$.

By consistency, we may take \(p>d\). Then \(u\in W_{\mathrm{loc}}^{2,p}(\mathbb{R}^d)\). Fix
$$
0<\sigma<\min\{1,\beta,1-d/p\}.
$$
The local Sobolev embedding gives \(u\in C_{\mathrm{loc}}^{1,\sigma}(\mathbb{R}^d)\). Since the coefficients are locally \(C^{0,\sigma}\), the equation
\[
\Delta u=
\frac{(\lambda+c|x|^\beta)u-b|x|^{\alpha-2}x\cdot\nabla u-f}
{1+|x|^\alpha}
\]
has a locally \(C^{0,\sigma}\) right-hand side. Interior Schauder regularity therefore yields \(u\in C_{\mathrm{loc}}^{2,\sigma}(\mathbb{R}^d)\subset C^2(\mathbb{R}^d)\); see \cite[Theorem~9.19]{GT}. By linearity, it suffices to consider real-valued \(f\), in which case \(u\) is real-valued because the coefficients are real and the resolvent solution is unique. Thus the maximum principle below applies to classical real-valued solutions.

Choose an integer $k\geq1$ such that
$
 2k>d+\beta+\alpha-1.
$
Direct differentiation gives
\begin{equation}\label{eq:range-barrier}
 \begin{aligned}
 &(\lambda-\Lambda)(1+|x|^2)^{-k}\\
 &\quad\geq
 \bigl(\lambda+c|x|^\beta-C_k(1+|x|^{\alpha-2})\bigr)
 (1+|x|^2)^{-k}.
 \end{aligned}
\end{equation}
Since $\beta>\alpha-2$, we can choose $R\geq1$ such that
$\operatorname{supp}f\subset B_R$ and
\[
 (\lambda-\Lambda)(1+|x|^2)^{-k}\geq0,
 \qquad |x|>R.
\]
Choose $M>0$ so that
\[
 |u(x)|\leq M(1+|x|^2)^{-k},
 \qquad |x|=R.
\]
Because $(\lambda-\Lambda)u=0$ for $|x|>R$, we have
\[
 (\lambda-\Lambda)
 \bigl(\pm u-M(1+|x|^2)^{-k}\bigr)\leq0,
 \qquad |x|>R.
\]
Both functions are nonpositive on $|x|=R$ and tend to zero at infinity. If either were positive somewhere, it would
attain a positive interior maximum. At that point, its gradient would vanish and its Laplacian would be nonpositive,
so its image under $\lambda-\Lambda$ would be strictly positive, since $\lambda+c|x|^\beta>0$. This is a contradiction.

Since $u$ is continuous on $\overline{B_R}$, increasing the constant gives
\begin{equation}\label{eq:resolvent-polynomial-decay}
|u(x)|\leq C_k(1+|x|^2)^{-k},
\qquad x\in\mathbb R^d.
\end{equation}

Outside $\operatorname{supp}f$, the resolvent equation becomes
\begin{equation}\label{eq:range-normalized}
\Delta u
+\frac{b|x|^{\alpha-2}x}{1+|x|^\alpha}\cdot\nabla u
=
\frac{\lambda+c|x|^\beta}{1+|x|^\alpha}\,u.
\end{equation}
The first-order coefficient is bounded, since
\[
\left|
\frac{b|x|^{\alpha-2}x}{1+|x|^\alpha}
\right|
=
|b|\frac{|x|^{\alpha-1}}{1+|x|^\alpha}
\leq |b|.
\]
If $|x|>R+2$, then $B(x,2)$ is disjoint from $\operatorname{supp}f$. Applying the interior $W^{2,p}$
estimate to \eqref{eq:range-normalized}, with its right-hand side treated as a forcing term, and using Sobolev embedding with $p>d$, we obtain
\begin{equation}\label{eq:resolvent-gradient}
\begin{aligned}
|\nabla u(x)|
&\leq C\|u\|_{W^{2,p}(B(x,1))}\\
&\leq C\left(
\|u\|_{L^p(B(x,2))}
+
\left\|
\frac{\lambda+c|\cdot|^\beta}{1+|\cdot|^\alpha}\,u
\right\|_{L^p(B(x,2))}
\right)\\
&\leq C_\lambda(1+|x|^\beta)
\sup_{y\in B(x,2)}|u(y)|.
\end{aligned}
\end{equation}
The constants are independent of $x$: the balls have fixed radii, the principal part is the Laplacian, and the
first-order coefficient is uniformly bounded; see \cite[Theorem~9.11]{GT}. For the last inequality, we used the fixed volume of $B(x,2)$ and the bound
\[
\frac{\lambda+c|y|^\beta}{1+|y|^\alpha}
\leq \lambda+c|y|^\beta
\leq C_\lambda(1+|x|^\beta),
\qquad y\in B(x,2).
\]

For $|x|\geq4$ and $y\in B(x,2)$, we have $|y|\geq|x|/2$. Hence \eqref{eq:resolvent-polynomial-decay} gives
\[
\sup_{y\in B(x,2)}|u(y)|
\leq C_k\left(1+\frac{|x|^2}{4}\right)^{-k}
\leq C(1+|x|)^{-2k}.
\]
Combining this with \eqref{eq:resolvent-gradient}, we obtain
\begin{equation}\label{eq:resolvent-gradient-decay}
|\nabla u(x)|
\leq C(1+|x|)^{\beta-2k}
\end{equation}
for all sufficiently large $|x|$.

We now check the required integrability. Since $2k>d+\beta$, estimate \eqref{eq:resolvent-polynomial-decay} implies
$
u\in L^1(\mathbb R^d),
$
and 
$
|x|^\beta u\in L^1(\mathbb R^d).
$
Moreover, \eqref{eq:resolvent-gradient-decay} gives, for all sufficiently large $R_0$,
\[
\int_{|x|>R_0}
(1+|x|^{\alpha-1})|\nabla u(x)|\,dx
\leq
C\int_{R_0}^\infty
r^{d-1+\beta+\alpha-1-2k}\,dr
<\infty,
\]
since $2k>d+\beta+\alpha-1$. Together with the local regularity of $u$, this yields
$
(1+|x|^{\alpha-1})\nabla u\in L^1(\mathbb R^d)
$
and hence 
$
u\in W^{1,1}(\mathbb R^d).
$

Finally, the resolvent equation gives
\[
(1+|x|^\alpha)\Delta u
=
(\lambda+c|x|^\beta)u
-b|x|^{\alpha-2}x\cdot\nabla u-f.
\]
Every term on the right belongs to $L^1$. In particular, the drift term is integrable because
\[
\bigl|b|x|^{\alpha-2}x\cdot\nabla u\bigr|
\leq |b|(1+|x|^{\alpha-1})|\nabla u|.
\]
Thus $(1+|x|^\alpha)\Delta u\in L^1(\mathbb R^d)$, and all the conditions defining $\mathcal D(\Lambda)$
are satisfied. Therefore $u\in\mathcal D(\Lambda)$.

\end{proof}

\section{Proofs of the main results}

\subsection{Proof of Theorem~\ref{thm:main}}
We divide the proof into three parts. We first prove generation by the Lumer--Phillips theorem and establish positivity through the resolvent. The domain estimates then follow from Proposition~\ref{closure}. Finally, we compare the resolvents to prove consistency with the $L^p$-semigroups.

\noindent
\emph{(i) Generation and positivity.}
By Proposition~\ref{closure}, the operator
$(\Lambda,\mathcal D(\Lambda))$ is closed and densely defined, with core $\Cc$. Passing to the limit in the norm inequality of Lemma~\ref{kato} gives
\begin{equation}\label{closed-diss}
 (\lambda-\omega_1)\|u\|_1
 \leq\|(\lambda-\Lambda)u\|_1,
 \qquad u\in\mathcal D(\Lambda),\quad\lambda>\omega_1.
\end{equation}
Thus $\Lambda-\omega_1$ is dissipative.

Fix $\lambda\geq\lambda_*$ and $1<p<\infty$. Given $f\in L^1$, choose $f_n\in\Cc$ such that
$f_n\to f$ in $L^1$. By Lemma~\ref{range},
\[
 u_n:=R(\lambda,\Lambda_p)f_n\in\mathcal D(\Lambda),
 \qquad
 (\lambda-\Lambda)u_n=f_n.
\]
Applying \eqref{closed-diss} to $u_n-u_m$, we obtain
\[
 \|u_n-u_m\|_1
 \leq\frac{\|f_n-f_m\|_1}{\lambda-\omega_1}.
\]
Hence $(u_n)$ converges in $L^1$ to some $u$, and
\[
 \Lambda u_n=\lambda u_n-f_n
 \longrightarrow\lambda u-f
 \quad\text{in }L^1.
\]
Since $\Lambda$ is closed, $ u\in\mathcal D(\Lambda)$ and $(\lambda-\Lambda)u=f$. Therefore
$$
 \operatorname{Ran}(\lambda-\Lambda)=L^1,
 \qquad\lambda\geq\lambda_*.
$$
The Lumer--Phillips theorem \cite[Theorem~II.3.15]{EN} shows that $\Lambda-\omega_1$ generates a contraction semigroup. Consequently, $\Lambda$ generates a strongly continuous semigroup $(T_1(t))_{t\geq0}$ satisfying
\eqref{semigroup-bound}, and its resolvent satisfies \eqref{resolvent-bound}.

If $f\geq0$, we may choose $f_n\geq0$. Lemma~\ref{range} then gives $u_n\geq0$, and passage to the limit in $L^1$ yields
\begin{equation}\label{large-resolvent-positive}
 R(\lambda,\Lambda)f\geq0,
 \qquad f\geq0,\quad\lambda\geq\lambda_*.
\end{equation}
For $t>0$ and $f\in L^1$, the exponential formula
\cite[Corollary~III.5.5]{EN} gives
\begin{equation}\label{L1-exponential-formula}
 T_1(t)f
 =\lim_{n\to\infty}
 \left[\frac nt
 R\left(\frac nt,\Lambda\right)\right]^n f
 \quad\text{in }L^1.
\end{equation}
For sufficiently large $n$, we have $n/t\geq\lambda_*$. Thus \eqref{large-resolvent-positive} and \eqref{L1-exponential-formula} imply $T_1(t)f\geq0$ whenever $f\geq0$.

Finally, for every real $\lambda>\omega_1$, the Laplace representation gives
\[
 R(\lambda,\Lambda)f
 =\int_0^\infty e^{-\lambda t}T_1(t)f\,dt
 \geq0,
 \qquad f\geq0,
\]
where the integral converges in $L^1$ by \eqref{semigroup-bound}.

\medskip
\noindent
\emph{(ii) Equivalence of norms.}
By \eqref{graph-estimate},
\[
 \|u\|_{\mathcal D(\Lambda)}
 \leq C\bigl(\|u\|_1+\|\Lambda u\|_1\bigr),
 \qquad u\in\mathcal D(\Lambda).
\]
Conversely, the expression for $\Lambda$ gives
\begin{align*}
 \|u\|_1+\|\Lambda u\|_1
 &\leq \|u\|_1
       +\|(1+|x|^\alpha)\Delta u\|_1\\
 &\quad+
       |b|\|(1+|x|^{\alpha-1})\nabla u\|_1
       +c\||x|^\beta u\|_1\\
 &\leq \max\{1,|b|,c\}\,
       \|u\|_{\mathcal D(\Lambda)}.
\end{align*}
Hence the domain norm and the graph norm are equivalent.

\medskip
\noindent
\emph{(iii) Consistency.}
Fix $1<p<\infty$. We first prove that
\begin{equation}\label{resolvent-consistency}
 R(\lambda,\Lambda)f=R(\lambda,\Lambda_p)f,
 \qquad
 f\in L^1\cap L^p,\quad\lambda\geq\lambda_*.
\end{equation}
For $f\in\Cc$, this follows from Lemma~\ref{range} and the injectivity of $\lambda-\Lambda$, which is a consequence of \eqref{closed-diss}.

For $f\in L^1\cap L^p$, choose $f_n\in\Cc$ such that
\[
 \|f_n-f\|_1+\|f_n-f\|_p\longrightarrow0.
\]
The boundedness of the resolvents gives
\[
 \begin{aligned}
 R(\lambda,\Lambda)f_n
 &\longrightarrow R(\lambda,\Lambda)f
 &&\text{in }L^1,\\
 R(\lambda,\Lambda_p)f_n
 &\longrightarrow R(\lambda,\Lambda_p)f
 &&\text{in }L^p.
 \end{aligned}
\]
The two sequences agree term by term, and both convergences hold in $L^1_{\mathrm{loc}}$.
Their limits therefore coincide, proving \eqref{resolvent-consistency}.

In particular, the common resolvent maps $L^1\cap L^p$ into itself. For $t>0$ and all sufficiently large $n$, so that $n/t\geq\lambda_*$,
\[
 \left[\frac nt
 R\left(\frac nt,\Lambda\right)\right]^n f
 =
 \left[\frac nt
 R\left(\frac nt,\Lambda_p\right)\right]^n f,
 \qquad f\in L^1\cap L^p.
\]
By \eqref{L1-exponential-formula} and its $L^p$ analogue, the two sides converge to $T_1(t)f$ in $L^1$ and to
$T_p(t)f$ in $L^p$, respectively. Identifying the limits in $L^1_{\mathrm{loc}}$ yields
\[
 T_1(t)f=T_p(t)f,
 \qquad f\in L^1\cap L^p,\quad t>0.
\]
The equality also holds at $t=0$, since both semigroups have initial value $I$.
\hfill\qed

\subsection{Proof of Corollary~\ref{cor:optimal-bound-compact-resolvent}}
We divide the proof into two parts. We first compute $\omega_1$, prove its optimality and deduce the contractivity criterion. We then combine local compactness with a uniform estimate outside large balls to prove that $\mathcal D(\Lambda)$ embeds compactly into $L^1(\mathbb R^d)$, and hence that
the resolvent is compact.

\noindent
\emph{(i) The optimal exponential bound.}
By definition,
\[
 \omega_1
 =\sup_{r\geq0}
 \bigl(\gamma r^{\alpha-2}-cr^\beta\bigr).
\]
If $\gamma\leq0$, the supremum is zero.
If $\gamma>0$, differentiation shows that the maximum
is attained at
\[
 r_*=
 \left(\frac{(\alpha-2)\gamma}{c\beta}\right)^{
 \frac{1}{\beta-\alpha+2}}.
\]
Since
$
 cr_*^\beta
 =\frac{\alpha-2}{\beta}\,\gamma r_*^{\alpha-2},
$
we obtain
\[
 \omega_1
 =\frac{\beta-\alpha+2}{\beta}\,
   \gamma r_*^{\alpha-2},
\]
which proves \eqref{omega1-explicit}.

Suppose now that \eqref{optimal-exponential-bound} holds for some $\omega\in\mathbb R$. Let $f\in\Cc$ satisfy $f\geq0$ and $\|f\|_1=1$. Since $f\in\mathcal D(\Lambda)$,
\[
 \frac{T_1(t)f-f}{t}
 \longrightarrow\Lambda f
 \quad\text{in }L^1
 \qquad(t\downarrow0).
\]
Positivity gives $\|T_1(t)f\|_1=\int_{\mathbb R^d}T_1(t)f\,dx$. Thus, integration by parts and
\eqref{optimal-exponential-bound} yield
\begin{align}
 \int_{\mathbb R^d}
 \bigl(\gamma|x|^{\alpha-2}-c|x|^\beta\bigr)f(x)\,dx
 &=\int_{\mathbb R^d}\Lambda f(x)\,dx \notag\\
 &=\lim_{t\downarrow0}
   \frac{\|T_1(t)f\|_1-1}{t} \notag\\
 &\leq\lim_{t\downarrow0}\frac{e^{\omega t}-1}{t}
 =\omega.
 \label{optimality-test}
\end{align}

For any $x\in\mathbb R^d$, choose such test functions supported in arbitrarily small balls about $x$. By continuity, \eqref{optimality-test} implies
\[
 \gamma|x|^{\alpha-2}-c|x|^\beta\leq\omega,
 \qquad x\in\mathbb R^d.
\]
Taking the supremum gives $\omega_1\leq\omega$. Since \eqref{semigroup-bound} holds, $\omega_1$ is optimal.

By \eqref{omega1-explicit}, we have $\omega_1\geq0$, with equality precisely when $b\geq\alpha$.
The contractivity criterion follows from this observation, optimality, and \eqref{semigroup-bound}.

\medskip
\noindent
\emph{(ii) Compactness.}
Let $(u_n)\subset\mathcal D(\Lambda)$ satisfy
\[
\sup_n\bigl(\|u_n\|_1+\|\Lambda u_n\|_1\bigr)<\infty.
\]
By \eqref{graph-estimate}, there exists $C>0$ such that
\[
\sup_n\left(
\|u_n\|_{W^{1,1}(\mathbb R^d)}
+\||x|^\beta u_n\|_1
\right)\leq C.
\]

For every integer $j\geq1$, the restrictions of $u_n$ to $B_j$ form a bounded sequence in $W^{1,1}(B_j)$.
The Rellich--Kondrachov theorem gives the compact embedding $W^{1,1}(B_j)\hookrightarrow L^1(B_j)$;
see \cite[Theorem~7.26]{GT}. We therefore extract a subsequence $(u_n^{(1)})$ converging in $L^1(B_1)$, then a subsequence $(u_n^{(2)})$ of $(u_n^{(1)})$ converging in $L^1(B_2)$, and continue in this way.
The diagonal subsequence
$$
v_n:=u_n^{(n)}
$$
converges in $L^1(B_j)$ for every fixed $j$, since its tail is a subsequence of $(u_n^{(j)})$. It consequently converges in $L^1(B_R)$ for every $R>0$, by choosing an integer $j>R$. We denote this subsequence again by $(u_n)$.

To pass from local to global convergence, we use the weighted bound. Since $|x|^\beta\geq R^\beta$ for
$|x|>R$, we have
\begin{equation}\label{uniform-tail}
\sup_n\int_{|x|>R}|u_n(x)|\,dx
\leq
R^{-\beta}\sup_n\||x|^\beta u_n\|_1
\leq CR^{-\beta},
\qquad R\geq1.
\end{equation}
Splitting the integral over $B_R$ and its complement gives
\[
\begin{aligned}
\|u_n-u_m\|_1
&\leq
\|u_n-u_m\|_{L^1(B_R)}
+\int_{|x|>R}|u_n(x)|\,dx
+\int_{|x|>R}|u_m(x)|\,dx\\
&\leq
\|u_n-u_m\|_{L^1(B_R)}+2CR^{-\beta}.
\end{aligned}
\]
Given $\varepsilon>0$, first choose $R\geq1$ so large that $2CR^{-\beta}<\varepsilon/2$, which is possible
because $\beta>0$. For this fixed $R$, convergence in $L^1(B_R)$ gives
an integer $N$ such that
\[
\|u_n-u_m\|_{L^1(B_R)}<\varepsilon/2
\qquad\text{for all }n,m\geq N.
\]
Hence $\|u_n-u_m\|_1<\varepsilon$ for $n,m\geq N$. Thus the extracted subsequence is Cauchy in $L^1(\mathbb R^d)$ and therefore converges in this space. This proves that the embedding of $\mathcal D(\Lambda)$,
equipped with the graph norm, into $L^1(\mathbb R^d)$ is compact.

Finally, fix $\lambda\in\rho(\Lambda)$.
For $f\in L^1(\mathbb R^d)$, the identity $ \Lambda R(\lambda,\Lambda)f =\lambda R(\lambda,\Lambda)f-f$ gives
\begin{align*}
\|R(\lambda,\Lambda)f\|_1
+\|\Lambda R(\lambda,\Lambda)f\|_1
&\leq
(1+|\lambda|)\|R(\lambda,\Lambda)f\|_1+\|f\|_1\\
&\leq C_\lambda\|f\|_1.
\end{align*}
Thus the resolvent maps bounded subsets of $L^1$ to bounded subsets of $\mathcal D(\Lambda)$ in the
graph norm. By the compactness of the embedding of $\mathcal D(\Lambda)$ into $L^1(\mathbb R^d)$, every bounded sequence $(f_n)$ in $L^1$ has a subsequence for which $R(\lambda,\Lambda)f_n$
converges in $L^1$. Therefore $R(\lambda,\Lambda)$ is compact on $L^1(\mathbb R^d)$.
\hfill\qed

\subsection{Proof of Theorem~\ref{thm:L1-analyticity}}

We divide the proof into five steps. In the first step, we establish a resolvent estimate for constant coefficients.
In the second and third steps, we use localization and weighted interpolation to obtain an estimate for $\Lambda$.
In the fourth step, we prove existence of the resolvent on a suitable sector and deduce analyticity. Finally, we prove compactness for positive times.

Fix 
$$\theta\in(\pi/2,\pi).$$

\medskip
\noindent
\emph{Step 1.}
Let $a_0,w_0>0$. Let $A$ denote the generator of the heat semigroup on
$L^1(\mathbb R^d)$. In particular,
\[
 C_c^\infty(\mathbb R^d)\subset D(A),
 \qquad Av=\Delta v
 \quad\text{for }v\in C_c^\infty(\mathbb R^d).
\]
The heat semigroup is bounded analytic of angle $\pi/2$; see \cite[Example~II.4.10]{EN}. The corresponding sectorial resolvent estimate therefore gives
\[
 \Sigma_\theta\subset\rho(A),
 \qquad
 \|R(\mu,A)\|_{L^1\to L^1}
 \leq \frac{M_\theta}{|\mu|},
 \qquad \mu\in\Sigma_\theta,
\]
where
\[
 \Sigma_\theta
 :=\{\mu\in\mathbb C\setminus\{0\}:|\arg\mu|<\theta\};
\]
see \cite[Theorem~II.4.6]{EN}. The constant $M_\theta$ depends only on $d$ and $\theta$. Now let $v\in C_c^\infty(\mathbb R^d)$.
Since $R(\mu,A)=(\mu-A)^{-1}$ and $Av=\Delta v$, we have
\[
 v=R(\mu,A)(\mu-A)v
  =R(\mu,A)(\mu-\Delta)v.
\]
Taking the $L^1$-norm and using the resolvent bound yields
\[
 \|v\|_1
 \leq \|R(\mu,A)\|_{L^1\to L^1}
       \|(\mu-\Delta)v\|_1
 \leq \frac{M_\theta}{|\mu|}
       \|(\mu-\Delta)v\|_1.
\]
Multiplying by $|\mu|$, we obtain
\[
 |\mu|\|v\|_1
 \leq M_\theta\|(\mu-\Delta)v\|_1,
 \qquad \mu\in\Sigma_\theta.
\]
Let $\lambda\in\Sigma_\theta$. Since $w_0>0$, we also have $(\lambda+w_0)/a_0\in\Sigma_\theta$.
Applying the preceding estimate with $\mu=(\lambda+w_0)/a_0$ and multiplying by $a_0$, we obtain
\[
 |\lambda+w_0|\|v\|_1
 \leq M_\theta
 \|(\lambda-a_0\Delta+w_0)v\|_1.
\]
Writing $\lambda=|\lambda|e^{i\varphi}$ with
$|\varphi|<\theta$, we obtain
\[
\begin{aligned}
|\lambda+w_0|
&=\left|e^{-i\varphi/2}(\lambda+w_0)\right|\\
&\geq \operatorname{Re}
       \left(e^{-i\varphi/2}(\lambda+w_0)\right)\\
&=(|\lambda|+w_0)\cos(\varphi/2)\\
&\geq (|\lambda|+w_0)\cos(\theta/2).
\end{aligned}
\]
Here we used $|e^{-i\varphi/2}|=1$ and $\cos(\varphi/2)\geq\cos(\theta/2)>0$. Consequently,
\begin{equation}\label{eq:an-constant-L1}
 (|\lambda|+w_0)\|v\|_1
 \leq C_\theta
 \|(\lambda-a_0\Delta+w_0)v\|_1,
\end{equation}
where $C_\theta=M_\theta/\cos(\theta/2)$.

To estimate $\Delta v$, use the identity
\[
 a_0\Delta v
 =(\lambda+w_0)v-(\lambda-a_0\Delta+w_0)v.
\]
Hence,
\[
 \begin{aligned}
 a_0\|\Delta v\|_1
 &\leq (|\lambda|+w_0)\|v\|_1
       +\|(\lambda-a_0\Delta+w_0)v\|_1\\
 &\leq (C_\theta+1)
       \|(\lambda-a_0\Delta+w_0)v\|_1.
 \end{aligned}
\]
Adding this inequality to \eqref{eq:an-constant-L1} yields
\begin{equation}\label{eq:an-frozen}
 (|\lambda|+w_0)\|v\|_1+a_0\|\Delta v\|_1
 \leq C_0\|(\lambda-a_0\Delta+w_0)v\|_1,
 \qquad v\in\Cc.
\end{equation}
The constant $C_0=2C_\theta+1$ depends only on $d$ and $\theta$, and is independent of $a_0$, $w_0$ and $\lambda$.

\medskip
\noindent
\emph{Step 2.}
We apply \eqref{eq:an-frozen} locally to $\lambda+1-\Lambda$, whose potential $1+c|x|^\beta$ is strictly positive. Fix $\delta\in(0,1/2)$. Using the covering result in \cite[Proposition~3]{boutiah et al} or in \cite[Proposition 6.1]{CF} and normalized smooth cutoffs, choose a locally finite partition of unity $(\chi_j)_{j\geq1}$ with compact supports
and uniformly bounded overlap, such that
\[
 \chi_j\geq0,\qquad \sum_j\chi_j=1.
\]
For suitable points $x_j\in\operatorname{supp}\chi_j$,
set
\[
 a_j:=1+|x_j|^\alpha,
 \qquad W_j:=1+c|x_j|^\beta.
\]
The partition may be chosen so that
\begin{equation}\label{eq:an-oscillation}
 \frac{|1+|x|^\alpha-a_j|}{1+|x|^\alpha}
 +\frac{|1+c|x|^\beta-W_j|}{1+c|x|^\beta}
 \leq\delta
 \quad\text{on }\operatorname{supp}\chi_j,
\end{equation}
and
\begin{equation}\label{eq:an-cutoffs}
 \sum_j|\nabla\chi_j(x)|
 \leq\frac{C_\delta}{1+|x|},
 \qquad
 \sum_j|\Delta\chi_j(x)|
 \leq\frac{C_\delta}{(1+|x|)^2}.
\end{equation}

Indeed, take covering balls whose radii are a sufficiently small fixed multiple of $1+|x_j|$. The logarithms of both coefficients are uniformly continuous as functions of $\log(1+|x|^2)$. Their oscillations on the doubled balls can therefore
be made smaller than $\log(1+\delta/2)$, which gives \eqref{eq:an-oscillation}. The scaled cutoffs, after normalization, satisfy \eqref{eq:an-cutoffs} by the bounded overlap and the comparability of $1+|x|$ and $1+|x_j|$ on these balls.

Let $u\in\Cc$ and $\lambda\in\Sigma_\theta$. Apply \eqref{eq:an-frozen} to $\chi_j u$ with $a_0=a_j$ and $w_0=W_j$. The product rule gives
\begin{equation}\label{eq:an-localized-identity}
\begin{aligned}
(\lambda-a_j\Delta+W_j)(\chi_j u)
={}&\chi_j(\lambda+1-\Lambda)u\\
&+b\chi_j|x|^{\alpha-2}x\cdot\nabla u\\
&+(1+|x|^\alpha-a_j)\chi_j\Delta u\\
&+(W_j-1-c|x|^\beta)\chi_j u\\
&-a_j\bigl(
 2\nabla\chi_j\cdot\nabla u+u\Delta\chi_j
 \bigr).
\end{aligned}
\end{equation}

Since the partition is nonnegative,
\[
 \sum_j\|\chi_j u\|_1=\|u\|_1,
 \qquad
 \sum_j\|\chi_j(\lambda+1-\Lambda)u\|_1
 =\|(\lambda+1-\Lambda)u\|_1.
\]
Moreover, $\Delta u=\sum_j\Delta(\chi_j u)$ and
\eqref{eq:an-oscillation} imply
\[
\begin{aligned}
 \|(1+|x|^\alpha)\Delta u\|_1
 &\leq\frac{1}{1-\delta}
       \sum_j a_j\|\Delta(\chi_j u)\|_1,\\
 \|(1+c|x|^\beta)u\|_1
 &\leq\frac{1}{1-\delta}
       \sum_j W_j\|\chi_j u\|_1.
\end{aligned}
\]

The coefficient errors in \eqref{eq:an-localized-identity} are bounded by $\delta$ times the corresponding weighted norms.
For the drift and cutoff terms, use \eqref{eq:an-cutoffs}, the comparison $a_j\leq(1+\delta)(1+|x|^\alpha)$ on each support,
and
\[
 \frac{1+|x|^\alpha}{1+|x|}
 \leq C(1+|x|^{\alpha-1}),
 \qquad
 \frac{1+|x|^\alpha}{(1+|x|)^2}
 \leq C(1+|x|^{\alpha-2}).
\]
Summing \eqref{eq:an-frozen} over $j$ therefore yields
\begin{equation}\label{eq:an-localization}
\begin{aligned}
&|\lambda|\|u\|_1
 +\|(1+|x|^\alpha)\Delta u\|_1
 +\|(1+c|x|^\beta)u\|_1\\
&\quad\leq C_1\|(\lambda+1-\Lambda)u\|_1\\
&\qquad+C_1\delta\bigl(
 \|(1+|x|^\alpha)\Delta u\|_1
 +\|(1+c|x|^\beta)u\|_1
 \bigr)\\
&\qquad+C_\delta\bigl(
 \|(1+|x|^{\alpha-1})\nabla u\|_1
 +\|(1+|x|^{\alpha-2})u\|_1
 \bigr).
\end{aligned}
\end{equation}
Since $(1-\delta)^{-1}\leq2$, we may take
$C_1=2C_0$, independently of $\delta$, $\lambda$
and $u$. The constant $C_\delta$ may depend on
$\delta$, but not on $\lambda$ or $u$.

\medskip
\noindent
\emph{Step 3.}
The weighted interpolation estimate established in the proof of Proposition~\ref{closure} states that, for every $\varepsilon>0$,
\begin{equation}\label{eq:an-interpolation}
\|(1+|x|^{\alpha-1})\nabla u\|_1
\leq
\varepsilon\|(1+|x|^\alpha)\Delta u\|_1
+C_\varepsilon\|(1+|x|^{\alpha-2})u\|_1.
\end{equation}
Since $\beta>\alpha-2$ and $c>0$, for every $\eta>0$
there exists $C_\eta>0$ such that
\[
1+|x|^{\alpha-2}
\leq\eta(1+c|x|^\beta)+C_\eta.
\]
Hence
\begin{equation}\label{eq:an-potential}
\|(1+|x|^{\alpha-2})u\|_1
\leq\eta\|(1+c|x|^\beta)u\|_1
+C_\eta\|u\|_1.
\end{equation}
Applying \eqref{eq:an-potential} with
$\eta=\varepsilon/(1+C_\varepsilon)$ and using
\eqref{eq:an-interpolation} gives
\begin{equation}\label{eq:an-lower-order}
\begin{aligned}
&\|(1+|x|^{\alpha-1})\nabla u\|_1
+\|(1+|x|^{\alpha-2})u\|_1\\
&\quad\leq
\varepsilon\bigl(
 \|(1+|x|^\alpha)\Delta u\|_1
+\|(1+c|x|^\beta)u\|_1
\bigr)
+C'_\varepsilon\|u\|_1.
\end{aligned}
\end{equation}

Choose first $\delta$ so that $C_1\delta\leq1/4$, and then $\varepsilon$ so that $C_\delta\varepsilon\leq1/4$.
Substituting \eqref{eq:an-lower-order} into \eqref{eq:an-localization} and absorbing the weighted terms yields
\[
\begin{aligned}
&|\lambda|\|u\|_1
+\frac12\bigl(
 \|(1+|x|^\alpha)\Delta u\|_1
+\|(1+c|x|^\beta)u\|_1
\bigr)\\
&\quad\leq
C_1\|(\lambda+1-\Lambda)u\|_1
+C_\delta C'_\varepsilon\|u\|_1.
\end{aligned}
\]
Taking
\[
R>\max\{2C_\delta C'_\varepsilon,\omega_1-1,0\},
\]
we conclude that
\begin{equation}\label{eq:an-coercivity}
|\lambda|\|u\|_1
\leq C\|(\lambda+1-\Lambda)u\|_1,
\qquad
\lambda\in\Sigma_\theta,\quad |\lambda|\geq R,
\end{equation}
with $C=2C_1$.

By Proposition~\ref{closure}, $\Cc$ is a core for $\Lambda$. Approximating in the graph norm therefore extends
\eqref{eq:an-coercivity} to every $u\in\mathcal D(\Lambda)$.

\medskip
\noindent
\emph{Step 4.}
Fix $r\geq R$. Since $1+r>\omega_1$,
Theorem~\ref{thm:main} gives $1+r\in\rho(\Lambda)$.
Hence the set
\[
 \bigl\{\varphi\in(-\theta,\theta):
 1+re^{i\varphi}\in\rho(\Lambda)\bigr\}
\]
contains zero. It is open because $\rho(\Lambda)$
is open.

We show that it is also relatively closed.
Let $\varphi_n\to\varphi\in(-\theta,\theta)$, with
$1+re^{i\varphi_n}\in\rho(\Lambda)$.
The estimate \eqref{eq:an-coercivity} gives
\[
 \|R(1+re^{i\varphi_n},\Lambda)\|_{\mathcal L(L^1)}
 \leq\frac{C}{r}.
\]
Therefore,
\[
\begin{aligned}
&\bigl\|
 r(e^{i\varphi}-e^{i\varphi_n})
 R(1+re^{i\varphi_n},\Lambda)
 \bigr\|_{\mathcal L(L^1)}\\
&\qquad\leq
 C|e^{i\varphi}-e^{i\varphi_n}|
 \longrightarrow0.
\end{aligned}
\]
For large $n$, this norm is less than one. The Neumann series then shows that $1+re^{i\varphi}-\Lambda$ is invertible.
Thus the set is relatively closed and, by connectedness, equals $(-\theta,\theta)$.

Since $r\geq R$ was arbitrary, we obtain $1+\lambda\in\rho(\Lambda)$, and 
\begin{equation}\label{eq:an-truncated-sector}
\|R(1+\lambda,\Lambda)\|_{\mathcal L(L^1)}
  \leq\frac{C}{|\lambda|},
  \qquad
  \lambda\in\Sigma_\theta,\quad |\lambda|\geq R.
\end{equation}
The resolvent bound follows again from \eqref{eq:an-coercivity}. We now choose
\[
 \omega>
 \max\left\{\omega_1,\,
 1+\frac{R}{\cos(\theta/2)}\right\}.
\]
Let $\lambda\in\omega+\Sigma_\theta$.
Since adding a positive real number preserves
$\Sigma_\theta$, we have
$
 \lambda-1=(\lambda-\omega)+(\omega-1) \in\Sigma_\theta.
$
Moreover, the preceding inequality gives
\[
\begin{aligned}
|\lambda-1|
&\geq \cos(\theta/2)
       \bigl(|\lambda-\omega|+\omega-1\bigr)\\
&\geq \cos(\theta/2)(\omega-1)>R,
\end{aligned}
\]
where the last inequality follows from the choice of $\omega$. Applying \eqref{eq:an-truncated-sector} with
$\lambda-1$ in place of $\lambda$, we conclude that
\[
 \omega+\Sigma_\theta\subset\rho(\Lambda),
 \qquad
 \|R(\lambda,\Lambda)\|_{\mathcal L(L^1)}
 \leq\frac{M}{|\lambda-\omega|},
\]
where $M=C/\cos(\theta/2)$. The preceding estimates show that $\Lambda-\omega I$ is sectorial in the sense of \cite[Definition~II.4.1]{EN}, with angle $\theta-\pi/2\in(0,\pi/2)$. By \cite[Theorem~II.4.6]{EN}, the semigroup $(e^{-\omega t}T_1(t))_{t\geq0}$ is therefore analytic. Multiplying its analytic extension by $e^{\omega z}$ gives an analytic extension of $(T_1(t))_{t\geq0}$.

\medskip
\noindent
\emph{Step 5.}
Since $\omega>\omega_1$, we have $\omega\in\rho(\Lambda)$. By Corollary~\ref{cor:optimal-bound-compact-resolvent}
\textup{(ii)}, $R(\omega,\Lambda)$ is compact on $L^1$. For every $t>0$, analyticity gives
\[
 T_1(t)L^1(\mathbb R^d)\subset\mathcal D(\Lambda),
 \qquad
 \Lambda T_1(t)\in\mathcal L(L^1).
\]
Consequently, $(\omega I-\Lambda)T_1(t)$ is bounded, and
\[
 T_1(t)
 =R(\omega,\Lambda)(\omega I-\Lambda)T_1(t).
\]
The right-hand side is the composition of a compact operator and a bounded operator. Hence $T_1(t)$ is compact for every $t>0$.
\hfill\qed

\makeatletter
\let\savedtocwrite\@tocwrite
\let\@tocwrite\@gobbletwo
\section*{AI usage disclosure}
\let\@tocwrite\savedtocwrite
\makeatother

The author used ChatGPT to assist in checking the calculations in the proof of Theorem ~\ref{thm:L1-analyticity}. We also used ChatGPT for editorial purposes, i.e.\,fixing typos, improving language, and carrying out the final editorial review of the paper. 

\bigskip

\noindent
\textbf{Data availability.} This article has no associated datasets.

\bigskip

\noindent
\textbf{Funding.} The author received no funding for this work.

\end{document}